\documentclass[a4paper,12pt, reqno]{amsart}
\usepackage{amsfonts, color}
\usepackage[textwidth=460pt]{geometry}
\usepackage{amstext, amsthm, amssymb, amsmath}
\usepackage{dsfont}
\usepackage{mathtools}
\usepackage{mathrsfs}
\usepackage{graphicx}
\usepackage{color}
\usepackage{lineno}

\usepackage{hyperref}

\usepackage{cleveref}
\hypersetup{colorlinks = true, citecolor = blue}

\theoremstyle{plain}
\begingroup
\newtheorem{theorem}{Theorem}[section]

\newtheorem{proposition}[theorem]{Proposition}

\endgroup

\theoremstyle{definition}
\begingroup
\newtheorem{definition}{Definition}

\endgroup

\theoremstyle{remark}
\newtheorem{remark}{Remark}

\numberwithin{equation}{section}

\newcommand{\Z}{{\mathbb Z}}
\newcommand{\R}{{\mathbb R}}
\newcommand{\N}{{\mathbb N}}
\newcommand{\T}{{\mathbb T}}
\newcommand{\C}{{\mathbb C}}

\newcommand{\J}{{\mathcal J}}
\newcommand{\U}{{\mathcal U}}
\newcommand{\de}{\, \mathrm{d}}

\usepackage[foot]{amsaddr}

\title[Approximation of diffeomorphisms for quantum state transfers]{Approximation of diffeomorphisms for quantum state transfers}

\author[E. Pozzoli]{Eugenio Pozzoli$^{1}$}
\email{eugenio.pozzoli@univ-rennes.fr}
\address{$^1$Univ Rennes, CNRS, IRMAR - UMR 6625, F-35000 Rennes, France.}

\author[A. Scagliotti]{Alessandro Scagliotti$^{2,3}$}
\email{scag@ma.tum.de}
\address{$^2$CIT School, Technical University of Munich, Garching bei M\"unchen, Germany.}
\address{$^3$Munich Center for Machine Learning (MCML), Munich, Germany.}

\begin{document}

\begin{abstract}
In this paper, we seek to combine two emerging standpoints in control theory.
On the one hand, recent advances in infinite-dimensional geometric control have unlocked a method for controlling (with arbitrary precision and in arbitrarily small times) state transfers for bilinear Schrödinger PDEs posed on a Riemannian manifold $M$. In particular, these arguments rely on controllability results in the group of the diffeomorphisms of $M$. 
On the other hand, using tools of $\Gamma$-convergence, it has been proved that we can phrase the retrieval of a diffeomorphism of $M$ as an ensemble optimal control problem.
More precisely, this is done by employing a control-affine system for \emph{simultaneously} steering a finite swarm of points towards the respective targets.
Here we blend these two theoretical approaches and numerically find control laws driving state transitions (such as eigenstate transfers) in small time in a bilinear Schrödinger PDE posed on the torus. Such systems have experimental relevance and are currently used to model rotational dynamics of molecules, and cold atoms trapped in periodic optical lattices. 
\end{abstract}

\maketitle

\section{Introduction}
\subsection{The Schrödinger equation}
In this work we study bilinear Schrödinger PDEs of the type
\begin{equation}\label{eq:schro}
\begin{cases}
i\partial_t\psi(t,x)=\left(-\Delta+V(x)+\sum_{j=1}^m u_j(t)W_j(x)\right)\psi(t,x), & \\
\psi(0,\cdot)=\psi_0,
\end{cases}
\end{equation}
with $ (t,x) \in (0,T) \times M$, where 
$M$ is a smooth connected boundaryless Riemannian manifold,
$\Delta$ is the Laplace-Beltrami operator of $M$,
the potentials $V, W_1,\dots,W_m\colon M \to \R$ 
and the controls $u_1,\dots,u_m\colon [0,T] \to \R$ are real valued functions. The uncontrolled operator $-\Delta+V$ is referred to as \emph{the drift}. Equation \eqref{eq:schro} describes an infinite-dimensional non-linear control-to-state system. By choosing a locally integrable control signal
$u=(u_1,\dots,u_m)$ and picking up an initial state $\psi_0$
in the unit sphere
\begin{equation} \label{def:S}
\mathcal{S}\coloneqq \{\psi\in L^2(M,\C)\, \mid \|\psi\|_{L^2}=1\},
\end{equation}
$\psi(t;u,\psi_0)$ denotes (if well-defined) the solution of \eqref{eq:schro} at time $t\geq 0$.

System \eqref{eq:schro} models a quantum particle on $M$, with free energy $-\Delta+V$, interacting with external potentials $W_j$, which can be switched on and off. This framework models a variety of physical situations, such as atoms in optical cavities \cite{haroche}, and molecular dynamics \cite{rabitz90}.

We are interested in the approximate controllability of (\ref{eq:schro}).
\begin{definition}
\begin{itemize}
\item The state $\psi_1\in \mathcal{S}$ is said to be approximately reachable from $\psi_0\in \mathcal{S}$ for \eqref{eq:schro}
if, for every $\varepsilon >0$, 
there exist 
a time $T\geq 0$,
a global phase $\theta \in [0,2\pi)$
and a 
control $u\in L^1_{\rm loc}([0,T],\R^m)$ such that 
$$\| \psi( T;u,\psi_0) - \psi_1 e^{i\theta} \|_{L^2} <\varepsilon.$$
\item The state $\psi_1\in \mathcal{S}$ is said to be small-time approximately reachable from $\psi_0\in \mathcal{S}$ for \eqref{eq:schro}
if, for every $\varepsilon >0$, 
there exist 
a time $\tau \in[0,\varepsilon]$,
a global phase $\theta \in [0,2\pi)$
and a 
control $u\in L^1_{\rm loc}([0,\tau],\R^m)$ such that 
$$\| \psi( \tau;u,\psi_0) - \psi_1 e^{i\theta} \|_{L^2} <\varepsilon.$$
\end{itemize}
\end{definition}
We shall also say that \eqref{eq:schro} is (small-time) approximately controllable if the set of (small-time) approximately reachable operators coincides with $\mathcal{S}$ from any $\psi_0\in\mathcal{S}$.

This problem, {and the problem of exact controllability in finite-dimensional systems}, has been widely investigated by control theorists in the last three decades, due to its relevance in applications to physics and chemistry (e.g. absorption spectroscopy, {and time-optimal pulse design in NMR \cite{khaneja-brockett-glaser}}), or computer science (e.g. error correction for robust quantum computation) \cite{preskill,victor}. E.g., if the drift $-\Delta+V$ has compact resolvent (hence purely point spectrum), \eqref{eq:schro} is known to be approximately controllable generically w.r.t. the potentials $V,W_1,\dots,W_m\in L^\infty(M,\R)$ \cite{MS-generic}. The first examples of small-time approximately controllable Schrödinger equations of the form \eqref{eq:schro} were obtained only recently in \cite{beauchard-Pozzoli}.

\medskip

In this article, we study a particular example of Schrödinger equation of the form \eqref{eq:schro}. 
\subsection{An equation posed on $M=\T=\R/2\pi\Z$.}

We consider a Schrödinger equation of the form \eqref{eq:schro},
\begin{align}
i\partial_t\psi(t,x)=&\big(-\partial^2_x+V(x)+u_1(t)\cos(x)+u_2(t)\sin(x)+\nonumber\\
&u_3(t)\cos(2x)+u_4(t)\cos(4x) \big)\psi(t,x),\label{eq:rotor}
\end{align}
where $V\in L^\infty(\T,\R)$.
This setting models the rotational motion of a rigid molecule whose dynamics are confined to a plane, controlled by electromagnetic fields (of constant directions and tunable amplitudes) coupled to four Fourier modes (see, e.g., the recent physical review \cite{koch} and the references therein). It also describes a (linear-in-state) Bose-Einstein condensate in an optical lattice with tunable depth and phase~\cite{sugny}, and two additional controls on higher Fourier modes. 

The approximate controllability of system \eqref{eq:rotor} was established in \cite{BCCS} with finite-dimensional Galerkin approximations and averaging techniques. The small-time approximate controllability of \eqref{eq:rotor}, which had remained an open question for a decade \cite{boussaid-caponigro-chambrion}, was also recently established in \cite{beauchard-pozzoli2} with completely different techniques, inspired by an infinite-dimensional geometric control approach introduced in \cite{duca-nersesyan}. 
\subsection{Our contribution}
Small-time approximate controllability of \eqref{eq:rotor} implies the theoretical capability of driving state transfers in arbitrarily small times (with the flaw of using of course controls with unbounded amplitudes). These results shed light on the theoretical and practical problem of quantum speed limit. 
Unfortunately, the technique developed in \cite{beauchard-pozzoli2} does not provide explicit control laws for achieving state transitions, but it only ensures the existence of such controls. In other words, in a crucial passage of the proof, the explicit character of the control is lost and only existence can be asserted. More precisely, this happens when invoking a celebrated diffeomorphism decomposition result. 
Let ${\rm Diff}^0(\T)$ be the connected component at the identity of the group of diffeomorphisms of $\T$.
Given a diffeomorphism $P\in {\rm Diff}^0(\T^d)$ isotopic to the identity, owing to the simplicity of the group  ${\rm Diff}^0(\T^d)$ established by Thurston \cite{thurston}, there exists $n\in \N$ and vector fields $f_1,\dots,f_n \in {\rm Vec}(\T^d)$ such that
$$P=\Phi^1_{f_n}\circ \cdots \circ \Phi^1_{f_1},$$
where $\Phi^t_f$ denotes the flow at time $t$ of the vector field $f$. In general, given $P \in {\rm Diff}^0(\T^d)$, it is highly nontrivial to guess the number $n$ and the vectors $f_1,\dots,f_n$ for which such a decomposition holds. 
Let us now assume that we are given a finite family of vector fields $g_1,\ldots,g_k \in {\rm Vec}(\T^d)$ whose flows composition generates approximately the group ${\rm Diff}^0(\T^d)$ (such a family exists, see e.g. \cite[Theorem 6.6]{agrachev-sarychev3}). Then, the problem of approximating $P$ with these flows turns out to be a control problem with infinite-dimensional state space, i.e., ${\rm Diff}^0(\T^d)$.

To overcome this issue, one can fix a finite family of vector fields $g_1,\ldots,g_k \in {\rm Vec}(\T)$ whose products of flows generate approximately the group ${\rm Diff}^0(\T)$ (such a family exists, see e.g. \cite[Theorem 6.6]{agrachev-sarychev3}). Then, the task of approximating $P$ with these flows is rephrased as a control problem with infinite-dimensional state space, i.e., ${\rm Diff}^0(\T)$. This standpoint has been adopted in several recent articles \cite{agrachev-sarychev3,scagliotti}. However, as we shall see in the statement of Theorem~\ref{thm:sm_time_app} below, when dealing with quantum systems, we do not address the problem of approximating $P$, but rather the problem of approximating in $L^2$ the target wavefunction $\det(DP)^{1/2}(\psi_0\circ P)$, where $\psi_0$ is the initial wavefunction. 

This may seem surprising, as the Schrödinger equation is not a transport equation. Nevertheless, a $L^2$-unitary transport mechanism is hidden in the Schrödinger equation and can be followed with specific control trajectories, as recently shown in \cite{beauchard-pozzoli2}, and as we recall for completeness in Section \ref{sec:diffeo_ctrl}.

Once the link with transport equations is made, our approximation task is then similar to the ensemble control problem, or the control problem for the Liouville equation, introduced in \cite{brockett} (for more recent theoretical and numerical studies on the subject we also refer to \cite{borzi,borzi2}). There, one looks for controls which are robust with respect to uncertainties on the initial condition. Such goal should not be confused with problems where one looks for controls robust with respect to dispersion in internal parameters of the system, also called ensemble control problems, e.g. in the NMR community \cite{khaneja-li,li}. 

The problem of approximating in $L^2$ the target wavefunction $\det(DP)^{1/2}(\psi_0\circ P)$, from any initial wavefunction $\psi_0\in L^2(\T,\C)$ and for any diffeomorphism $P\in {\rm Diff}^0(\T)$ was theoretically solved in \cite{beauchard-pozzoli2}. However, the control laws were not explicit. 

In this article, our aim is to numerically compensate for the lack of constructiveness of \cite{beauchard-pozzoli2} by showing that the explicit character of the control needed to steer the system towards the state
$\det(DP)^{1/2}(\psi_0\circ P)$ can be recovered through an approximation procedure.

The idea of using flows of control-affine systems for diffeomorphism approximation was proposed by \cite{agrachev-sarychev3}, and it has been a cornerstone for recent developements on the mathematical foundations of Machine Learning \cite{scagliotti2023normalizing,agrachev_letrouit}.
Here, we pursue an approach similar to that proposed in \cite{scagliotti}.
That is, the diffeomorphism approximation task was formulated as an infinite-dimensional optimal control problem in ${\rm Diff}^0(\R^d)$. Moreover, aiming to have a numerically tractable problem, in \cite{scagliotti} the author employed $\Gamma$-convergence arguments for reducing the original infinite-dimensional state space setting to an optimal control problem concerning the simultaneous steering of a finite (large enough) ensemble of points.

For the sake of simplicity, here we design explicit control laws for specific but relevant state transfers of \eqref{eq:rotor} (namely, from the ground to one of the first excited states, and from the ground to a highly oscillating state), posed on $\T$. However, similar methods could in theory be applied for arbitrary state transfers, and for similar Schrödinger equations posed on $\T^d, d> 1$.

A natural further question concerns the genericity of the control potentials $W_1,\dots,W_m$ leading to small-time approximate controllability (such a genericity problem was solved for large-time approximate controllability in \cite{MS-generic}, and was also recently studied in the context of neural networks \cite{agrachev_letrouit}): this will be the subject of future investigation.

\subsection{Structure of the paper}
The paper is organised as follows: in Section \ref{sec:diffeo_ctrl} we recall the relationship between small-time approximate controllability of the Schrödinger equation \eqref{eq:rotor} and the approximate controllability of ${\rm Diff}^0(\T)$. This section is based on the results obtained in \cite{beauchard-pozzoli2}. In Section \ref{sec:Ensemble_OC} we pose the problem of approximating a given diffeomorphism as an ensemble optimal control problem. This section is based on the results obtained in \cite{scagliotti}. In Section \ref{sec:numerics} we present the original contribution of this paper, that is, we numerically obtain explicit controls for driving specific quantum state transfers.

\section{From controlled Schrödinger equation to diffeomorphisms control} \label{sec:diffeo_ctrl}
We consider the following product formula for $\tau>0$ and $\varphi\in C^\infty(\T,\R)$
$$e^{i\frac{\varphi}{2\tau}}e^{i(\Delta-V)}e^{-i\frac{\varphi}{2\tau}}=\exp\left(i\tau(\Delta-V)+\mathcal{T}_{\nabla \varphi}+i\frac{|\nabla\varphi|^2}{4\tau}\right), $$
where $\mathcal{T}_f$ is a first-order skew-adjoint differential operator on $L^2(\T,\C)$ with domain $H^1(\T,\C)$, acting as
$\mathcal{T}_f(\psi) \coloneqq \langle f,\nabla \psi\rangle+\frac{1}{2}{\rm div}(f)\psi$.
Hence, by the Trotter-Kato product formula \cite[Theorem VIII.31]{rs2}, we have that 
\begin{equation} \label{strong_limits}
\left( 
e^{i \frac{|\nabla \varphi|^2}{4 n\tau}  }
e^{i  \frac{\varphi}{2 \tau} }
e^{i\frac{\tau}{n}(\Delta-V)}
e^{-i \frac{\varphi}{2 \tau} }
\right)^{n}
 \underset{n \to \infty}{\longrightarrow} 
e^{i \tau (\Delta-V) + \mathcal{T}_{\nabla \varphi}}
\end{equation}
strongly (i.e., tested on any $\psi_0\in L^2(\T,\C)$, and in the $L^2$-norm). Moreover, owing to \cite[Theorem VIII.21 \& Theorem VIII.25(a)]{rs1}, the following strong convergence holds as well:
$$e^{i \tau (\Delta-V) + \mathcal{T}_{\nabla\varphi}}
\quad \underset{\tau \to 0}{\longrightarrow} \quad
e^{ \mathcal{T}_{\nabla\varphi} }.$$
Then, we consider $\varphi(x)=\cos(x)$, and we notice that
$|\nabla \varphi|^2=\sin^2(x)=\frac{1-\cos(2x)}{2}$. 
Analogously, we take $\varphi(x)=\sin(x)$ and we observe that
$|\nabla \varphi|^2=\cos^2(x)=\frac{1+\cos(2x)}{2}$.
Since $\cos(x),\cos(2x),\cos(4x)$ are all control operators appearing in \eqref{eq:rotor}, {the limits \eqref{strong_limits} with $\varphi=\cos(x)$ (hence $|\nabla \varphi|^2=-\cos(2x)/2$ modulo additive constants), or $\varphi=\sin(x)$ (hence $|\nabla \varphi|^2=\cos(2x)/2$ modulo additive constants), or $\varphi=\cos(2x)$ (hence $|\nabla \varphi|^2=-2\cos(4x)$ modulo additive constants) are all approximate small-time dynamics of \eqref{eq:rotor}}, and we obtain that the following states  
\begin{equation} \label{eq:STAR1} e^{\mathcal{T}_{u\sin(x)\partial_x}}\psi_0,e^{\mathcal{T}_{v\cos(x)\partial_x}}\psi_0,
e^{\mathcal{T}_{w\sin(2x)\partial_x}}\psi_0, \quad u,v,w\in\R,
\end{equation}
are small-time approximately reachable from any $\psi_0\in L^2(\T,\C)$. 
Therefore, by the Trotter-Kato product formula, 
\begin{equation}\label{eq:bracket1}
\left(e^{\frac{-1}{tn}\mathcal{T}_f}e^{-t\mathcal{T}_g}  e^{\frac{1}{tn}\mathcal{T}_f} e^{t\mathcal{T}_g}  \right)^n
\underset{n \to \infty}{\longrightarrow }
\exp\Big(\frac{-1}{t}\mathcal{T}_f+ e^{-t\mathcal{T}_g}  \frac{1}{t}\mathcal{T}_f e^{t\mathcal{T}_g} \Big)
\end{equation}
and by \cite[Lemma 26]{beauchard-pozzoli2}
\begin{equation}\label{eq:bracket2}
\exp\left(\frac{-1}{t}\mathcal{T}_f+ e^{-t\mathcal{T}_g}  \frac{1}{t}\mathcal{T}_f e^{t\mathcal{T}_g} \right)
 \underset{t \to 0}{\longrightarrow } 
e^{\mathcal{T}_{[f,g]}},
\end{equation}
where the above limits are in the strong sense. Since 
$$[\cos(x)\partial_x,\sin(x)\partial_x]=\partial_x,$$
we also get that the states
\begin{equation}\label{eq:STAR2}
e^{\mathcal{T}_{s\partial_x}}\psi_0, \quad s\in\R, 
\end{equation}
are small-time approximately reachable from any $\psi_0$. Notice that, by the method of characteristics and Liouville formula,
\begin{align*}
\left( e^{t \mathcal{T}_f} \psi \right)(x)= 
|D_x\Phi_f^t(x)|^{1/2} \, \psi\big( 
\Phi_f^t(x)
\big) = e^{\frac{1}{2} \int_0^t {\rm div} f( 
\Phi_f^s(x)
) \de s} \, \psi \big( 
\Phi_f^t(x)
\big),
\end{align*}
where we used the abbreviation $|D_x\Phi_f^t(x)|\coloneqq \det D_x\Phi_f^t(x)$ for the determinant of the Jacobian.
At this point, it is convenient to define the following fields
\begin{equation} \label{eq:def_controlled_fields}
    g_0(x) \coloneqq \partial_x, \,\, g_1(x) \coloneqq \sin(x)\partial_x, \,\, g_2(x) \coloneqq \sin(2x)\partial_x
\end{equation}
for every $x\in\T$.
The following result was obtained in \cite[Theorems 13 \& 15]{beauchard-pozzoli2}, for a slightly different system (that is, \eqref{eq:rotor} with $u_3\equiv u_4\equiv 0$). Thanks to the presence of the additional control operators $\cos(2x),\cos(4x)$, here we can furnish a simplified proof.

\begin{theorem} \label{thm:sm_time_app}
For any $P\in {\rm Diff}^0(\T),\psi_0\in L^2(\T,\C)$, the state $|DP|^{1/2}(\psi_0\circ P)$ is small-time approximately reachable for \eqref{eq:rotor} from $\psi_0$.
\end{theorem}
\begin{proof}
{We showed in \eqref{eq:STAR1} and \eqref{eq:STAR2} that $e^{\mathcal{T}_{ug_i}}\psi_0, u\in \R, i=1,2,3$, are small-time approximately reachable states. Hence, using \eqref{eq:bracket1} and \eqref{eq:bracket2}, for any $f\in {\rm Lie}\{g_0,g_1,g_2\}$, the state $e^{\mathcal{T}_f}\psi_0$ is small-time approximately reachable.} By \cite[Theorem 6.6]{agrachev-sarychev3}, for any $f\in {\rm Vec}(\T)$ there exists 
$\{f_n\}_{n\in\N}\subset {\rm Lie}\{g_0,g_1,g_2\} $
such that
$\Phi_{f_n}^1 \to  \Phi_f^1$
in $C^\infty$ as $n \to \infty$. By the dominated convergence theorem in $L^2$, we get
$e^{\mathcal{T}_{f_n}}\psi_0 \to e^{\mathcal{T}_{f}}\psi_0$ as $n \to \infty$.
{Hence, for any $f\in {\rm Vec}(\T)$ and any $\psi_0\in L^2(\T,\C)$, the state $e^{\mathcal{T}_f}\psi_0$ is small-time approximately reachable.} 
By \cite{thurston}, there exists $k\in \N$ and vector fields $f_1,\dots,f_k\in {\rm Vec}(\T)$ such that
$P=\Phi^1_{f_k} \circ \cdots \circ \Phi^1_{f_1}$.
{Thus, we have
$$|DP|^{1/2}(\psi_0\circ P)=e^{\mathcal{T}_{f_n}}\dots e^{\mathcal{T}_{f_1}}\psi_0,$$
and we conclude that $|DP|^{1/2}(\psi_0\circ P)$ is small-time approximately reachable.}
\end{proof}
The previous proof also shows that we can approximate in $L^2$ the state $|DP|^{1/2}(\psi_0\circ P)$ with products of the form 
$$e^{\mathcal{T}_{f_n}}\cdots e^{\mathcal{T}_{f_1}}\psi_0,$$
where 
$f_j\in \{ g_0,g_1,g_2\}$. 
Hence, we have reduced our task to the approximation of a given $P\in {\rm Diff}^0(\T)$ through product of flows of the form
$$\Phi^1_{f_n} \circ \cdots \circ \Phi^1_{f_1},$$
where 
$f_j\in \{ g_0,g_1,g_2\}$.
We shall explain how to numerically tackle this point in the forthcoming sections.
\begin{remark}\label{rmk:moser}
Note that, by decomposing the wavefunction in polar coordinates $\psi=\rho e^{i\varphi},$ with radial part $\rho:=|\psi|\geq 0, \rho\in L^2(\T,[0,\infty)), \|\rho\|_{L^2}=1,$ and angular part $\varphi\in L^2(\T,\R)$, we can control $\psi$ by separately controlling $\rho$ and $\varphi$. The control of the angular part can be performed (approximately and in arbitrarily small times) with explicit controls, as done in \cite{duca-nersesyan} (see also \cite{chambrion-pozzoli,duca-pozzoli} for the control on the angular part {in Schrödinger bilinear PDEs posed, resp., on the $2$-dimensional sphere and on the euclidean space of arbitrary dimension}). We are thus left to find explicit controls for steering the radial part $\rho=|\psi|$. Note also that, thanks to a celebrated theorem of Moser \cite{moser}, for any couple of initial and final states $\rho_0,\rho_1\in C^\infty(\T,(0,\infty)))$ such that $\|\rho_0\|_{L^2}=\|\rho_1\|_{L^2}$, there exists $P\in {\rm Diff}^0(\T)$ such that $|DP|^{1/2}(\rho_0\circ P)=\rho_1$. By density and approximation, for any couple $\rho_0,\rho_1\in L^2(\T,[0,\infty))$ such that $\|\rho_0\|_{L^2}=\|\rho_1\|_{L^2}$, and any $\varepsilon>0$, there exists $P\in {\rm Diff}^0(\T)$ such that $\||DP|^{1/2}(\rho_0\circ P)-\rho_1\|_{L^2}<\varepsilon$.
\end{remark}
\section{Ensemble optimal control formulation} \label{sec:Ensemble_OC}

In this section, we introduce the framework that we employ later for the numerical construction of the quantum state transfer {$\psi_0\to | DP|^{1/2}(\psi_0\circ P)$}.
The building block of our approach is the
following controlled system on $\T$:
\begin{equation} \label{eq:ctrl_ODE}
    \dot x_u(t) =  \sum_{i=0}^2 g_i\big(x_u(t)\big) u_i(t) \quad \mbox{for a.e. } t\in [0,1], 
\end{equation}
where $g_0,g_1,g_2$ are the vector fields introduced in \eqref{eq:def_controlled_fields}, and $t\mapsto u(t)=\big(u_0(t), u_1(t), u_2(t) \big)$ is the control signal used to steer the system. 
Here, we allow the control $u$ to vary in $\U\subseteq L^2([0,1],\R^3)$, where $\U$ is a weakly closed subspace.
For every $u\in \U$ and for every initial condition $x_u(0)=x_0\in \T$, the corresponding solution of \eqref{eq:ctrl_ODE} is absolutely continuous and is defined throughout the evolution interval $[0,1]$. For every $u\in \U$, we denote with $\Phi_u^1\colon \T\to\T$ the flow at time $t=1$ induced by the time-varying vector field $\sum_{i=0}^2g_iu_i$.
Then, given $\psi_0,\psi_1\in \mathcal{S}$, we consider the functional $\J^\alpha \colon \U \to \R$ defined as follows
\begin{equation} \label{eq:def_funct}
    \begin{split}
        \J^\alpha(u) \coloneqq \frac{\alpha}{2}\|u\|_{L^2}^2  + \min_{\theta\in[0,2 \pi)} 
        \| |D \Phi_u^1|^{1/2} (\psi_0\circ \Phi_u^1) - e^{i\theta}\psi_1 \|_{\mathcal{H}}^2
    \end{split}
\end{equation}
for every $u\in\U$, where $\alpha>0$ tunes the $L^2$-regularization.
With a classical argument involving the direct method of calculus of variations, it is possible to show that $\J^\alpha$ admits minimizers (see \cite[Thm.~3.2]{Scag23}).
In the next result we relate the values of the parameter $\alpha$ to the quality of the state transfer achieved by the  minimizers of $\J^\alpha$. 

\begin{proposition} \label{prop:argmin}
    Let $\J^\alpha$ be defined as in \eqref{eq:def_funct} with $\U=L^2$, and let us assume that there exists $P\in {\rm Diff}^0(\T)$ such that $\psi_1 = (DP)^{1/2} (\psi_0\circ P)$. Moreover, for every $\alpha>0$ let us define
    \begin{equation*}
        \begin{split}
            \kappa(\alpha) \coloneqq  \sup \bigg\{ 
        \min_{\theta\in[0,2 \pi)} 
        \| |D \Phi_{u^\star}^1|^{1/2} (\psi_0\circ \Phi_{u^\star}^1) - e^{i\theta}\psi_1 \|_{\mathcal{H}}^2 : 
        {u^\star} \in \arg\min \J^\alpha
        \bigg\}.
        \end{split}
    \end{equation*}
    Then, we have that $\lim_{\alpha\to 0}\kappa(\alpha)=0$.
\end{proposition}
\begin{proof}
    Owing to Theorem~\ref{thm:sm_time_app}, for every $\epsilon>0$ there exists $\hat u \in \U$ such that
    $$\min_{\theta\in[0,2 \pi)} \| |D \Phi_{\hat u}^1|^{1/2} (\psi_0\circ \Phi_{\hat u}^1) - e^{i\theta}\psi_1 \|_{\mathcal{H}}^2\leq \frac\epsilon2. $$
    Hence, for every $\alpha\leq \bar\alpha\coloneqq \epsilon/\|\hat u\|_{L^2}^2$, we obtain that
    \begin{equation*}
        \begin{split}
            \min_{\theta\in[0,2 \pi)} 
        \| |D \Phi_{u^\star}^1|^{1/2} (\psi_0\circ \Phi_{u^\star}^1) - e^{i\theta}\psi_1 \|_{\mathcal{H}}^2 \leq \J^\alpha(u^\star) 
         \leq \J^\alpha(\hat u) \leq 
        \frac\epsilon2 + \frac\alpha2 \|\hat u\|_{L^2}^2 \leq \epsilon
        \end{split}
    \end{equation*}
    for every $$u^\star\in \arg\min \J^\alpha,$$ yielding $\kappa(\alpha)\leq \epsilon$ for every $\alpha\in (0,\bar\alpha)$.
\end{proof}

\begin{remark} \label{rmk:non_reach}
    The previous result can be extended also to cases when the target state $\psi_1\in\mathcal{S}$ is not approximately reachable, or when $\U \not =  L^2$. Namely, in such a scenario, we obtain that
    \begin{equation*}
        \begin{split}
            \lim_{\alpha\to 0} \kappa(\alpha) 
            = \inf \bigg\{ 
        \min_{\theta\in[0,2 \pi)} 
         \| |D \Phi_{u}^1|^{1/2}  (\psi_0\circ \Phi_{u}^1) - e^{i\theta}\psi_1 \|_{\mathcal{H}}^2 :  u \in \U
        \bigg\},
        \end{split}
    \end{equation*}
    where the quantity at the right-hand side is bigger than $0$ and provides a lower bound on the approximation error in such a situation.
    \end{remark}
    \begin{remark}
    In view of Proposition~\ref{prop:argmin}, on the one hand, setting $\alpha$ as small as possible sounds a desirable option, since the lower is its value, the better is the state transfer achieved by the corresponding minimizers of $\J^\alpha$. On the other hand, the coercivity of $\J^\alpha$---which is crucial for the existence of minimizers---relies on the penalization on the $L^2$-norm of the control. For this reason, when $\alpha\ll 1$, the minimization problem gets harder, both from the theoretical and the numerical perspective.
\end{remark}

\begin{remark}
    As recently shown in~\cite[Corollary~3.3]{trade-off_invariance}, excluding at most countably many exceptional values of $\alpha$, we have that 
    \begin{equation} \label{eq:trade-off}
        \kappa(\alpha)= \min_{\theta\in[0,2 \pi)} 
        \| |D \Phi_{u^\star}^1|^{1/2} (\psi_0\circ \Phi_{u^\star}^1) - e^{i\theta}\psi_1 \|_{\mathcal{H}}^2
    \end{equation}
    \emph{for every} $u^\star\in \arg\min \J^\alpha$. In other words, if $\alpha$ is not exceptional, the right-hand side of \eqref{eq:trade-off} is constant as $u^\star$ varies in $\arg\min \J^\alpha$.
\end{remark}

The result in Proposition~\ref{prop:argmin} paves the way for the construction of approximated state transfers through the minimization of $\J^\alpha$.
However, each of its evaluations requires the computation of the transfer error (i.e., the second term at the right-hand side of \eqref{eq:def_funct}), which in turn needs the resolution of \eqref{eq:ctrl_ODE} for every $x_u(0)=x_0$ as $x_0$ ranges in $\T$. 
Since solving infinitely many Cauchy problems is clearly unfeasible in practice, this point might potentially be a major obstacle for a numerical approach.
Fortunately, this issue has already been addressed in ensemble control \cite{Scag23,Scag25} by taking advantage of $\Gamma$-convergence tools.
More precisely, inspired by the approach followed by \cite{scagliotti} in the Euclidean setting, if $\psi\in \mathcal{S} \cap C^0(\T,\C)$, we make the following approximation:
\begin{equation*}
    \begin{split}
        \| \psi \|_{\mathcal{H}}^2 = 
    \int_\T |\psi(x)|^2 \de \mu(x)
    \approx 
    \int_\T |\psi(x)|^2 \de \mu_N(x)
    = \frac1N \sum_{j=1}^N |\psi(x_j)|^2,
    \end{split}
\end{equation*}
where $\mu$ is the Haar measure on $\T$, $\mu_N \coloneqq  \frac1N \sum_{j=1}^N \delta_{x_j}$, and $\{x_1,\ldots,x_N\}$ is a lattice of $N$ equi-spaced points in $\T$. 
Then, we define the functional 
$\J^\alpha_N \colon \U \to \R$ as follows
\begin{equation} \label{eq:def_funct_N}
    \begin{split}
        \J^\alpha_N(u) \coloneqq \frac{\alpha}{2}\|u\|_{L^2}^2  \quad +\min_{\theta\in[0,2 \pi)}
        \int_\T \big||D \Phi_u^1|^{1/2} (\psi_0\circ \Phi_u^1) - e^{i\theta}\psi_1 \big|^2 \de \mu_N,
    \end{split}
\end{equation}
for every $u\in\U$. Similarly as for $\J^\alpha$, weak lower semi-continuity in $L^2$ and coercivity yield $\arg \min \J^\alpha_N \neq \emptyset$. 
At this point, is natural to wonder if the minimizers of $\J^\alpha_N$ attain quasi-optimal values when used to evaluate the original functional $\J^\alpha$. In this regard, $\Gamma$-convergence provides a positive answer.
We recall that the sequence of functionals $(\J^\alpha_N)_{N\geq1}$ is $\Gamma$-convergent to $\J^\alpha$ as $N\to\infty$ with respect to the $L^2$-weak topology if: 
\begin{enumerate}
    \item[a)] For every $u\in \U$ and for every sequence $(u_N)_{N\geq1}$ such that $u_N\rightharpoonup_{L^2} u$ as $N\to\infty$, we have $\J^\alpha(u)\leq\liminf_{N\to\infty}\J^\alpha_N(u_N)$.
    \item[b)] For every $u\in \U$, there exists a sequence $(\tilde u_N)_{N\geq1}$ such that $\tilde u_N\rightharpoonup_{L^2} u$ as $N\to\infty$ and $\J^\alpha(u)=\lim_{N\to\infty}\J^\alpha_N(\tilde u_N)$.
\end{enumerate}
For a thorough introduction to this subject, we recommend the monograph \cite{D93}.

\begin{theorem}[$\Gamma$-convergence] \label{thm:G_conv}
    Let us consider $\psi_0,\psi_1 \in \mathcal{S} \cap C^0(\T,\C)$, and let $\J^\alpha$ and $\J^\alpha_N$ be the functionals defined in \eqref{eq:def_funct} and \eqref{eq:def_funct_N}, respectively.
    Then, for every $\alpha>0$, the family of functionals $(\J^\alpha_N)_{N\geq1}$ is $\Gamma$-convergent to $\J^\alpha$ as $N\to\infty$ with respect to the $L^2$-weak topology.
    Consequently, we have that:
    \begin{itemize}
        \item Given any sequence $(u_N^\star)_{N\geq1} \subset \U$ such that $u_N^\star\in \arg\min \J_N^\alpha$ for every $N\geq 1$, it follows that $(u_N^\star)_{N\geq1}$ is \emph{$L^2$-strongly pre-compact}.
        \item If $(u^\star_{N_j})_{j\geq 1}$ is a subsequence such that $\|u^\star_{N_j} - u^\star\|_{L^2}\to 0$ as $j\to\infty$, then $u^\star\in\arg \min \J^\alpha$.
        \item Finally, we have that
    \begin{equation*}
        \lim_{N\to\infty}  
        \J^\alpha_N(u_N^\star) =
        \lim_{N\to\infty} \min_\U \J^\alpha_N 
        = \min_\U \J^\alpha.
    \end{equation*}
    \end{itemize}
\end{theorem}
\begin{proof}
    This result follows from \cite[Thm.~4.6 and Cor.~4.8]{Scag23} as a particular case.
\end{proof}
\begin{remark}
    We stress the fact that in Theorem~\ref{thm:G_conv} the continuity of $\psi_0,\psi_1$ is crucial. Moreover, the fact that $\alpha>0$ is needed for the equi-coercivity of the functionals $(\J^\alpha_N)_{N\geq1}$.
\end{remark}
The results reported in Theorem~\ref{thm:sm_time_app} provide the theoretical backbone for the numerical experiments of the forthcoming section.

\section{Numerical experiments} \label{sec:numerics}

In this section, we take advantage of the reformulation of the state transfer as a diffeomorphism approximation task (Section~\ref{sec:diffeo_ctrl}), and of the infinite-to-finite dimension reduction made possible by $\Gamma$-convergence (Section~\ref{sec:Ensemble_OC}).
Blending these standpoints, we propose a numerical construction for achieving an approximate state transfer.

\subsection{Evolution approximation}
We consider the control system \eqref{eq:ctrl_ODE} over the time horizon $[0,1]$.
We divide the evolution interval into $M=961$ equi-spaced nodes $t_0=0,\ldots, t_M=1$, with $t_{k}-t_{k-1}=2^{-6}/15$ for every $k=1,\ldots,M$.
In the experiments, we introduce the finite-dimensional subspace $\U^M \subset L^2([0,1],\R^3)$, where
\begin{equation*}
    \begin{split}
        \U^M\coloneqq \big\{ 
    u\in L^2([0,1],\R^3) \mid 
    \,\, \forall k=1,\ldots,M, \,\exists c_k: u(s)=c_k \, \forall s\in [t_{k-1},t_k]
    \big\}.
    \end{split}
\end{equation*}
Since the elements in $\U^M$ are piecewise constant on pre-determined intervals, we denote with $u= (u_k)_{k=1,\ldots,M} =(u_{k,i})_{k=1,\ldots,M}^{i=0,1,2}$ the elements of $\U^M$. 
For expositional convenience, we define the mapping $G\colon \T\times \R^3\to \R$ as
\begin{equation*}
    G(x,w) \coloneqq \sum_{i=0}^2 w_ig_i(x)
\end{equation*}
for every $(x,w)\in \T\times \R^3$.
We discretize the dynamics in \eqref{eq:ctrl_ODE} using the explicit Euler scheme, i.e., for every assigned initial position $x_0$, we compute
\begin{equation} \label{eq:discr_evol}
    x_{k} = x_{k-1} + h G(x_{k-1},u_k)
\end{equation}
for every $k=1,\ldots,M$. In \eqref{eq:discr_evol} we set $h=2^{-6}/15$.
For every $u\in\U^M$, we denote with $F_u^1\colon \T\to \T$ the discrete terminal-time flow induced by \eqref{eq:discr_evol} and corresponding to $u$.

\subsection{Numerical computation of state transfers}
Following the steps described in Section~\ref{sec:Ensemble_OC}, we construct $\mu_N \coloneqq  \frac1N \sum_{j=1}^N \delta_{x_j}$, where $N=628$ and $\{x_1,\ldots,x_N\}$ is a lattice of equi-spaced points with size $0.01$.
In the experiments, we consider the ground state $\psi_0 \equiv \frac{1}{\sqrt{2\pi}}$.
Moreover, the target quantum states $\psi_1\in \mathcal{S} \cap C^0(\T,[0,\infty))$ that we choose \emph{take values in $\R_+$}. In view of Remark \ref{rmk:moser}, this assumption is not restrictive, and the problem of approximating $\psi_1$ as $|DP|^{1/2}(\psi_0\circ P)$ is well-posed.
The objective functional $\J_N^\alpha\colon \U^M\to \R$ related to the discretized ensemble optimal control problem has the form: 
\begin{equation} \label{eq:obj_funct_num}
    \begin{split}
        \J^\alpha_N(u) \coloneqq \frac{\alpha}{2}\|u\|_{L^2}^2  \qquad +
        \int_\T \big||D F_u^1|^{1/2} (\psi_0\circ F_u^1) - \psi_1 \big|^2 \de \mu_N,
    \end{split}
\end{equation}
where $|DF(x)|=\det DF(x)$. We used $\alpha=10^{-7}$.
\begin{remark}
    If we compare \eqref{eq:def_funct_N} and \eqref{eq:obj_funct_num}, we observe that in the latter, the minimum over $[0,2\pi)$ accounting for a global phase has disappeared.
    This is due to the fact that we shall work with initial and target states $\psi_0,\psi_1$ valued in $\R_+$. Hence, the minimum on the global phase is always attained for $\theta=0$.
\end{remark}

We solved the optimization problem on Python, and we minimized \eqref{eq:obj_funct_num} using {the Damped-BFGS scheme (see \cite[Procedure~18.3]{nocedal1999numerical}), with an adaptive choice of the descent step-size based on Wolfe's conditions (see \cite[Section~3.1]{nocedal1999numerical})}. We empolyed the automatic differentiation tools of Pytorch.
We ran the scripts on a MacBookPro$^{\mbox{\textregistered}}$  with $16$~GB of RAM and with Apple~M1~Pro CPU.

We studied the state transfer problem for two different target states.
Namely, we first set
\begin{equation} \label{eq:psi_cos_2x}
    \psi_1'(x) = \frac{1}{\sqrt{\pi}} \big| \cos(x) \big|,
\end{equation}
and then
\begin{equation} \label{eq:psi_cos_3x}
    \psi_1''(x) = \frac{1}{\sqrt{2\pi}} \sqrt{1 + \frac{4}{5}\cos(3x)}.
\end{equation}
In both cases, we started with a random initial guess where the components of $u=(u_{k,i})_{k=1,\ldots,M}^{i=0,1,2}$ were sampled as i.i.d.~realizations of the gaussian distribution $\mathcal{N}(0,0.01)$.
We report the respective results in Fig.~\ref{fig:transfer_1} and Fig.~\ref{fig:transfer_2}. In the figures, the top pictures represent the graphs of $\psi_0$ (black), of the target state (blue), and of the state obtained through the transformation induced by the computed control (magenta), while the bottom images report the profiles of the computed controls responsible for the two state transfers.


    \begin{figure}
        \centering
        \includegraphics[width= 0.49\linewidth]{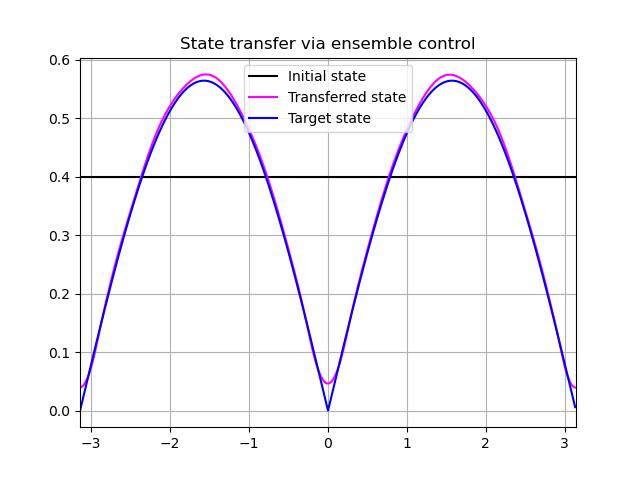}
        \includegraphics[width= 0.49\linewidth]{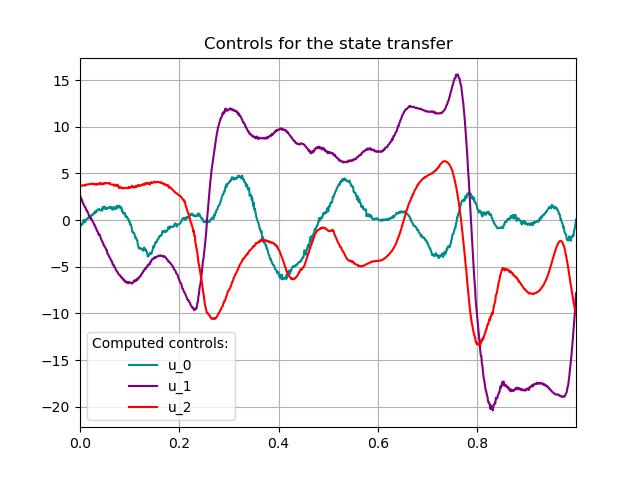}
        \caption{State transfer from the ground state $\psi_0 \equiv \frac{1}{\sqrt{2\pi}}$ to the target state $\psi_1'$ (see \eqref{eq:psi_cos_2x}). We used a random initialization for the control $u\in \U^M$, and we performed $\sim 2.2\cdot 10^3$ iterations of the Damped-BFGS scheme ($\sim15$~mins of computational time). The resulting relative mismatch in the $L^2$-norm amounts at $0.0225$.}
        \label{fig:transfer_1}
    \end{figure}

    \begin{figure}
        \centering
        \includegraphics[width= 0.49\linewidth]{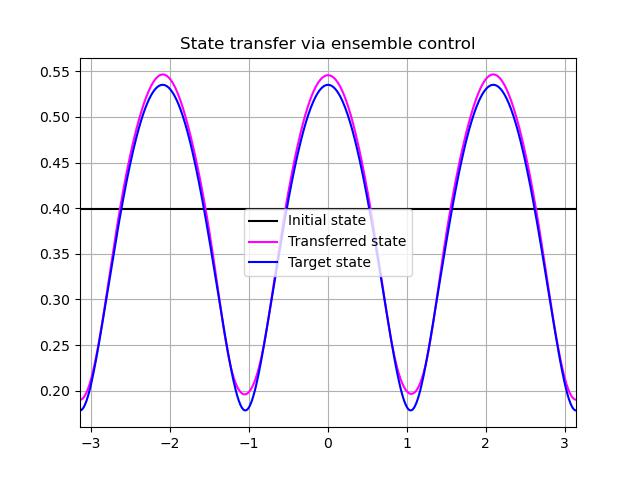}
        \includegraphics[width= 0.49\linewidth]{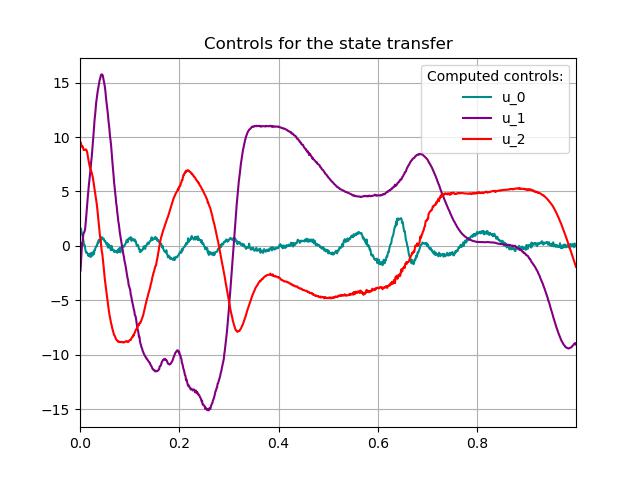}
        \caption{State transfer from the ground state $\psi_0 \equiv \frac{1}{\sqrt{2\pi}}$ to the target state $\psi_1''$ (see \eqref{eq:psi_cos_3x}). We used a random initialization for the control $u\in \U^M$, and we performed $2.8\cdot 10^3$ iterations of the the Damped-BFGS scheme ($\sim20$~mins of computational time). The resulting relative mismatch in the $L^2$-norm amounts at $0.0227$.}
        \label{fig:transfer_2}
    \end{figure}


\section*{Conclusions}
In this paper, we proposed a method for constructing explicit control laws steering approximate state transfers in a bilinear Schrödinger PDE posed on the torus. Such systems are experimentally tested in physics and represent powerful platforms for quantum technology. Indeed, they model e.g.~the orientational motion of rigid molecules, as well as the dynamics of cold atoms trapped in periodic optical lattices (known as Bose-Einstein condensates).

Even though small-time approximate controllability results have been recently established, determining explicitly such controls is, in general, challenging.

In the first part of the paper, we used geometric control arguments to rephrase the state transfer as a task of diffeomorphisms approximation.
Then, relying on the powerful tools of $\Gamma$-convergence, we reduced the infinite-dimensional diffeomorphism approximation to a finite-ensemble control problem.

Finally, by taking advantage of this bridging, we performed some experiments and found numerically explicit controls that achieve approximate state transfers.

\section*{Acknowledgments}

This research has been funded in whole or in part by the French National Research Agency (ANR) as part of the QuBiCCS project ”ANR-24-CE40-3008-01”.
This project has received financial support from the CNRS through the MITI interdisciplinary programs.

A.S.~acknowledges partial support from INdAM-GNAMPA.

Finally, the authors would like to thank the organizers of the conference ``Frontiers in sub-Riemannian geometry" (held at CIRM, Marseille, France, in November 2024) for the stimulating environment, and the CIRM for the kind hospitality, where some of the ideas of this work were conceived.

\bibliographystyle{abbrv}
\bibliography{references}

\end{document}